\documentclass{amsart}
\usepackage{amssymb,amsmath,bbm}
\usepackage{textcomp}
\usepackage{url}

\newcommand{\RR}{\mathbb{R}}

\newcommand{\HH}{\mathcal{H}}
\newcommand{\VV}{\mathcal{V}}
\newcommand{\rank}{\textup{rank}}

\newcommand{\convexhull}{\textup{conv}}
\newcommand{\conichull}{\textup{cone}}

\newcommand{\aff}{\textup{aff}}
\newcommand{\linearhull}{\textup{lin}}

\DeclareMathOperator {\rowsOp}{rows}
\DeclareMathOperator {\colsOp}{cols}

\newcommand \rows[1]{\rowsOp(#1)}
\newcommand \cols[1]{\colsOp(#1)}
\newcommand \setDef[2]{\{#1:#2\}}
\newcommand \inc[1]{{#1}_{\textup{inc}}}
\newcommand{\zerovec}{\mathbb{O}}
\newcommand{\onevec}{\mathbbm{1}}
\newcommand{\lin}{\textup{lin}}
\newcommand{\lineal}{\textup{lineal}}
\newcommand{\leftkernel}{\textup{leftkernel}}

\newcommand{\setdef}[2]{\{{#1}\,:\,{#2}\}}
\newcommand{\unitvec}[1]{\mathbbm{e}_{#1}}

\newtheorem{theorem}{Theorem}

\newtheorem{lemma}[theorem]{Lemma}

\newtheorem{corollary}[theorem]{Corollary}
\newtheorem{proposition}[theorem]{Proposition}
\theoremstyle{definition}

\newtheorem{example}[theorem]{Example}

\theoremstyle{remark}

\title{Which Nonnegative Matrices are Slack Matrices?}

\author{Jo{\~a}o Gouveia}
\address{CMUC, Department of Mathematics,
  University of Coimbra, 3001-454 Coimbra, Portugal}
\email{jgouveia@mat.uc.pt}

\author{Roland Grappe}
\address{Universit\'e Paris 13, Sorbonne Paris Cit\'e, LIPN, CNRS, (UMR 7030), F-93430, Villetaneuse, France} 
\email{Roland.Grappe@lipn.univ-paris13.fr}

\author{Volker Kaibel}
\address{Otto-von-Guericke Universit\"at Magdeburg, Fakut\"at f\"ur Mathematik, 39106 Magdeburg, Germany} \email{kaibel@ovgu.de}

\author{Kanstantsin Pashkovich}
\address{Dipartimento di Matematica, Universit\`a degli Studi di Padova, Via Trieste 63, 35121 Padova, Italy} 
\email{pashkovich@math.unipd.it}

\author{Richard Z. Robinson}
\address{Department of Mathematics, University of Washington, Box
  354350, Seattle, WA 98195, USA} \email{rzr@uw.edu}

\author{Rekha R. Thomas}
\address{Department of Mathematics, University of Washington, Box
  354350, Seattle, WA 98195, USA} \email{rrthomas@uw.edu}

\usepackage{fancyhdr}
\thanks{Gouveia was supported by by the Centre for Mathematics at the University of Coimbra and Fundac\~ao para a Ci\^encia e a
Tecnologia, through the European program COMPETE/FEDER, Kaibel by the Deutsche Forschungsgemeinschaft (KA 1616/4-1), Pashkovich by the Progetto di Eccellenza 2008-2009 of the Fondazione Cassa Risparmio di Padova e Rovigo, Robinson by
the U.S. National Science Foundation Graduate Research Fellowship (DGE-1256082),
and Thomas by the U.S. National Science Foundation grant DMS-1115293.}

\begin{document}

\begin{abstract}
In this paper we characterize the slack matrices of cones and polytopes among all nonnegative matrices.
This leads to an algorithm for deciding whether a given matrix is a slack matrix. The underlying
decision problem is equivalent to the {\em polyhedral verification problem} whose complexity
is unknown.
\end{abstract}

\maketitle

\section{Introduction}
This paper is concerned with a class of nonnegative matrices with real entries, called {\em slack matrices}, that arise naturally from polyhedral cones and polytopes. Given a polytope $P \subset \RR^n$ with vertices $v_1, \ldots, v_p$ and facet inequalities $a_j^Tx \leq \beta_j$ for $j=1,\ldots,q$, a {\em slack matrix} of $P$ is the $p \times q$ nonnegative matrix whose $(i,j)$-entry is $\beta_j - a_j^Tv_i$, the {\em slack} (distance from equality), of the $i$th vertex $v_i$ in the $j$th facet inequality $a_j^Tx \leq \beta_j$ of $P$. A similar definition holds for polyhedral cones.

Slack matrices form an interesting class of nonnegative matrices with many special properties. Most obviously, if $M$ is a slack matrix of a polytope $P$, then the zeros in $M$ record the face lattice of $P$ and hence the combinatorial structure of $P$. In its entirety, $M$ specifies an embedding of $P$ up to affine transformation. However, slack matrices carry much more  surprising information about $P$. In \cite{Yannakakis}, Yannakakis proved that the {\em nonnegative rank} of a slack matrix of $P$ is the minimum $k$ such that $P$ is the linear image of an affine slice of the positive orthant $\RR^k_+$. 
We use $\RR_+$ to denote the set of nonnegative real numbers. The nonnegative rank of a matrix $M \in \RR^{p \times q}_+$ is the smallest $k$ such there there exists vectors $a_1, \ldots, a_p \in \RR^k_+$ and $b_1, \ldots, b_q \in \RR^k_+$ such that $M_{ij} = a_i^Tb_j$.
Affine slices of positive orthants that project onto $P$ are called {\em polyhedral lifts} or {\em polyhedral extended formulations} of $P$ and the smallest $k$ such that $\RR^k_+$ admits a lift of $P$ is called the {\em (polyhedral) extension complexity} or nonnegative rank of $P$. If the extension complexity of $P$ is small (polynomial in the dimension of $P$), then usually it is possible to optimize a linear 
function over $P$ in polynomial time by optimizing an appropriate function on the lift. This is a powerful technique in optimization 
that yields polynomial time algorithms for linear optimization over complicated polytopes. There are many instances of 
$n$-dimensional polytopes with exponentially many (in $n$) facets that allow small polyhedral lifts. 

Yannakakis' result was generalized in \cite{GPT2012} to lifts of convex sets by affine slices of convex cones via {\em cone factorizations} of {\em slack operators}. Even in the larger context of cone lifts of convex sets, the case of polytopes is the simplest and the theory relies on slack matrices of polytopes and their factorizations through cones. Thus, understanding the structure of these matrices is fundamental for this theory. There are several phenomena that occur in the class of nonnegative matrices that have not yet been observed for slack matrices. For instance, an important open question is whether there exists a family of slack matrices of polytopes that exhibit an exponential gap between nonnegative rank and {\em positive semidefinite rank}. (If ${\mathcal S}^k_+$ denotes the cone of $k \times k$ real symmetric positive semidefinite matrices, then the positive semidefinite rank of a matrix $M \in \RR^{p \times q}_+$ is the smallest $k$ such that there exists matrices $A_i \in {\mathcal S}^k_+, \,\,i=1,\ldots,p$ and $B_j \in {\mathcal S}^k_+, \,\,j=1,\ldots,q$ such that $M_{ij} = \langle A_i, B_j \rangle$.) While there are simple families of matrices that exhibit even arbitrarily large gaps between nonnegative and positive semidefinite ranks \cite[Example 5]{GPT2012}, no family of slack matrices with this property is known. Such a family would be a clear witness for the power of {\em semidefinite programming} over linear programming in lifts of polytopes.

This paper was motivated by the many open questions about slack matrices which rely on understanding the structure of these matrices. We establish two main characterizations of slack matrices of polyhedral cones and polytopes.
In Section~\ref{sec:geom} we establish linear algebraic characterizations: Theorem~\ref{thm:slackconesrows} for cones and Theorem~\ref{thm:sm poly} for polytopes. In Section~\ref{sec:comb} we give combinatorial characterizations: Theorem~\ref{thm:incidences_polytopes} for polytopes and Theorem~\ref{thm:incidences_cones} for polyhedral cones. In Section~\ref{sec:algorithm} we use our characterization from Section~\ref{sec:geom} to give an algorithm for recognizing slack matrices. The computational complexity of this problem is unknown and is equivalent to the {\em polyhedral verification problem}. There are several further geometric and complexity results about slack matrices throughout the paper.

{\bf Notation}:  For a set of vectors $\mathcal{A} = \{ a_1,\ldots,a_p \}$, $\conichull(\mathcal{A}) := \{ \sum \lambda_i a_i \,:\, \lambda_i \geq 0 \}$ is the cone spanned by $\mathcal{A}$;  $\convexhull(\mathcal{A}) := \{ \sum \lambda_i a_i \,:\, 
\lambda_i \geq 0, \sum \lambda_i = 1 \}$ is the convex hull of $\mathcal{A}$; 
$\linearhull(\mathcal{A}) := \{ \sum \lambda_i a_i \,:\, \lambda_i \in \RR \}$ is the linear span of $\mathcal{A}$, and 
$\aff(\mathcal{A}) := \{ \sum \lambda_i a_i \,:\, \sum \lambda_i = 1 \}$ is the affine span of $\mathcal{A}$. The above sets can also be defined for an infinite subset $\mathcal{A} \subset \RR^n$ by taking unions over all finite subsets of $\mathcal{A}$.
For a $n \times q$ matrix $M$, we let $\rows{M}$ and $\cols{M}$ denote the sets of all rows and columns, respectively, of $M$.  We let $\mathcal{A} \cdot M$ be the set of vectors $\setdef{x^T M}{x \in \mathcal{A}}$.  For a set $K \subset \RR^n$, $\lineal(K)$ is the largest subspace contained in $K$, known as the {\em lineality space} of $K$. The dimension of a polytope $P$, $\dim(P)$ is the dimension of $\aff(P)$, the affine hull of $P$, and the dimension of a cone $K$ is the dimension of $\lin(K)$.

\section{Geometric Characterizations of Slack Matrices} \label{sec:geom}

\subsection{Slack Matrices of Polyhedral Cones} \label{subsec:cones}

Consider the polyhedral cone 
\begin{equation*}
	K=\setdef{x\in\RR^n}{x^TB\ge\zerovec} = \RR_+^p\cdot A 
\end{equation*}
in $\RR^n$ constrained by the columns of the matrix $B\in\RR^{n\times q}$ and generated by the rows of the matrix $A\in\RR^{p\times n}$. 
We call (the set of rows of) $A$ a \emph{$\mathcal{V}$-representation} and (the set of columns of) $B$ an \emph{$\mathcal{H}$-representation} of~$K$. 
The {\em slack matrix} of~$K$ with respect to the representation  $(A,B)$ is $S=AB\in\RR_+^{p\times q}$. Its $(i,j)$-entry records the ``slack'' of the $i$th generator of $K$ with respect to the $j$th inequality of $K$ in the given description of $K$.

Let $\mathcal{S}_K$ denote the set of all slack matrices of~$K$. For $S \in \mathcal{S}_K$, any matrix obtained by scaling 
the rows and columns of $S$ by positive reals is again in $\mathcal{S}_K$ since scaling the vectors in a $\mathcal{V}$ and/or $\mathcal{H}$-representation of $K$ does not change $K$. Also, $\mathcal{S}_K$ can have matrices of different sizes as adding redundant inequalities and/or generators to the representations of $K$ does not change $K$. 
From 
\begin{multline*}
	(\RR^n\cdot B) \cap\RR_+^q
	=K\cdot B
	=(\RR_+^p\cdot A)\cdot B
	=\RR_+^p\cdot S\\
	\subseteq(\RR^p\cdot S)\cap\RR_+^q
	=(\RR^p\cdot A B)\cap\RR_+^q
	\subseteq(\RR^n\cdot B)\cap\RR_+^q
\end{multline*}
we find that 
	$\RR_+^p\cdot S=\RR^p\cdot S\cap\RR_+^q$
which says that the {\em cone generated by the rows of $S$} coincides with the {\em nonnegative part of the row span of $S$}.
In fact, this relation characterizes slack matrices of cones:

\begin{theorem}\label{thm:slackconesrows}
	A nonnegative matrix $M\in\RR_+^{p\times q}$ is a slack matrix of a polyhedral cone if and only if
	\begin{equation}\label{eq:RCGC}
		\RR_+^p\cdot M=\RR^p\cdot M\cap\RR_+^q,
	\end{equation}     
	or in other words, the cone spanned by the rows of $M$ coincides with the nonnegative part of the row span of $M$.  
\end{theorem}

\begin{proof}
	It remains to show that every matrix $M\in\RR_+^{p\times q}$ with
		$\RR_+^p\cdot M=\RR^p\cdot M\cap\RR_+^q$
	is a slack matrix of some cone. Let $n=\rank(M)$ and choose a bijective linear map
	\begin{equation*}
		\varphi\,:\,\RR^p\cdot M\rightarrow\RR^n
	\end{equation*}
	that preserves the (standard) scalar product (an isometry). Let $M_i$ denote the $i$th row of $M$ and let $A \in \RR^{p \times n}$ be the matrix whose rows are $\varphi(M_i)$. Let $\pi: \RR^q \rightarrow \RR^p \cdot M$ be an orthogonal projection and let  $B \in \RR^{n \times q}$ be the matrix whose columns 
	are $\varphi(\pi(\unitvec{1})),\dots\varphi(\pi(\unitvec{q}))$ where $\unitvec{i}$ is the $i$th standard unit vector in~$\RR^q$. Then $M=AB$ and 
	using~\eqref{eq:RCGC}, 
	\begin{multline*}
		K = \setdef{x\in\RR^n}{x^TB\ge\zerovec}
		=\setdef{\varphi(y)}{y\in\RR^p\cdot M, \,\,\varphi(y)^TB\ge\zerovec}\\
		=\varphi(\RR^p\cdot M\cap\RR_+^q)=\varphi(\RR_+^p\cdot M)=\RR_+^p\cdot A\,,
	\end{multline*}
	which shows that~$M$ is a slack matrix of the cone $K$.
\end{proof}

Recall that the dual cone of $K$ is the cone 
\begin{equation*}
	K^{\star}=\setdef{y\in\RR^n}{x^Ty\ge 0\text{ for all }x\in K}
	=\setdef{y\in\RR^n}{Ay\ge 0} = B \cdot \RR^q_+.
\end{equation*}
Hence, $S^T$ is a slack matrix of~$K^{\star}$ and we get the following result. 

\begin{proposition} \label{prop:transpose for cones}
	A nonnegative real matrix is a slack matrix of a polyhedral cone if and only if its transpose is also the slack matrix of a polyhedral cone.
\end{proposition}

In particular, we obtain the following consequence of Theorem~\ref{thm:slackconesrows}.
\begin{corollary}\label{cor:slackconescols}
	A nonnegative matrix $M\in\RR_+^{p\times q}$ is a slack matrix of a polyhedral cone if and only if
	\begin{equation}\label{eq:CCGC}
		M\cdot\RR_+^q=M\cdot\RR^q\cap\RR_+^p,  
	\end{equation}
	or in other words, the cone spanned by the columns of $M$ coincides with the nonnegative part of the column span of $M$.
\end{corollary}

We say that a matrix~$M$ satisfies the \emph{row cone generating condition} (\emph{RCGC})  if~\eqref{eq:RCGC} holds and the \emph{column cone generating condition} (\emph{CCGC})  if~\eqref{eq:CCGC} holds. 

\begin{corollary}\label{cor:sm ccgc rcgc}
	For a nonnegative matrix $M\in\RR_+^{p\times q}$ the following statements are pairwise equivalent:
	\begin{itemize}
		\item $M$ is a slack matrix of a polyhedral cone.
		\item $M$ satisfies the RCGC.
		\item $M$ satisfies the CCGC.
	\end{itemize}
\end{corollary}
The equivalence of RCGC and CCGC for a general nonnegative matrix is not obvious. However, its proof becomes transparent via the theory of slack matrices of polyhedral cones and cone duality.

For a nonnegative $M$ with RCGC/CCGC, the proof of Theorem~\ref{thm:slackconesrows} showed how to produce a cone $K$ such that $M \in \mathcal{S}_K$. We now give another way to produce such a cone $K$ that will be useful later.
For any matrix $M \in \RR^{p \times q}$ of rank $k$, we will call a factorization of the form $M=AB$ with 
$A \in \RR^{p \times k}$, $\rank(A) = k$ and $B \in \RR^{k \times q}$, $\rank(B)=k$ a {\em rank factorization} of $M$. 

\begin{lemma} \label{lem:rank factorization}   
Let $M \in \RR^{p \times q}_+$
 be the slack matrix of a polyhedral cone and let $M=AB$ be a rank factorization of $M$. Then if $K$ is the cone generated by the rows of $A$, the columns of $B$ form an $\HH$-representation of $K$. In particular, $M \in \mathcal{S}_K$.
\end{lemma}

\begin{proof}
Let $K=\conichull (\{a_1,\ldots,a_p\})$ and $\widetilde{K}= \{ x \in \RR^k \,:\, x^Tb^j \geq 0, \,\,j=1,\ldots,q \}$ where $a_i$ is the $i$th row of $A$ and $b^j$ is the $j$th column of $B$.  We need to show that $K=\widetilde{K}$.
Since $M$ is a nonnegative matrix, $K \subseteq \widetilde{K}$. To prove the reverse inclusion we will argue that every linear function that is nonnegative on $K$ is also nonnegative on $\widetilde{K}$.
Let $k=\rank(M)$. Since $M$ has the CCGC and $A \cdot \RR^k = M \cdot \RR^q$ (since $\rank(A)=k$), we have that 
$M \cdot \RR^q_+ =  A \cdot \RR^k \cap \RR^p_+$.
 Suppose $L(x) = \ell_1x_1 + \ldots + \ell_k x_k  \geq 0$ for all $x \in K$.  Then the evaluation vector $(L(a_1),\ldots,L(a_p))^T = A \ell$ lies in $M \cdot \RR^q_+$.  Since the columns of $M$ are $Ab^j$ for $j=1,\ldots,q$, there exists $\lambda_j \geq 0$ such that
$A \ell = \sum_{j=1}^{q} \lambda_j (Ab^j) = A \sum_{j=1}^{q} \lambda_j b^j$. This implies that $\ell = \sum_{j=1}^{q} \lambda_j b^j$ since the columns of $A$ are linearly independent. Therefore, $\ell$ is a nonnegative linear combination of the $b^j$'s and $L(x) \geq 0$ is valid on $\widetilde{K}$.
\end{proof}

\subsection{Slack Matrices of Polytopes} \label{subsec:polytopes}

We now investigate the slack matrices of polytopes. Let $V\in\RR^{p\times n}$ and 
$P=\convexhull(\rows{V})$ be the polytope in $\RR^n$ that is the convex hull of the rows of $V$. Suppose also that 
$P=\setdef{x\in\RR^n}{Wx\le w}$ with $W\in\RR^{q\times n}$ and $w\in\RR^q$. To avoid unnecessary inconveniences, we assume that $\dim(P)\ge 1$. We call (the set of rows of) $V$ a $\mathcal{V}$-representation and (the set of columns of) $[w,-W]^T$ an $\mathcal{H}$-representation of~$P$. 
The {\em slack matrix} of~$P$ with respect to the representation  $(V,W,w)$ is then 
\begin{equation}\label{eq:defslackpoly}
	S=[\onevec,V]
		\cdot
		[w,-W]^T
	\in\RR_+^{p\times q}\,.
\end{equation}

We denote the set of all slack matrices of~$P$ by $\mathcal{S}_P$. Clearly, scaling the columns of a slack matrix of~$P$ by positive scalars yields another slack matrix of~$P$, because scaling the vectors in an $\mathcal{H}$-representation of $P$ yields another $\mathcal{H}$-representation of $P$. However, we cannot scale the 
rows of a matrix $S \in \mathcal{S}_P$ and still stay in $\mathcal{S}_P$.

The matrix~$S$  is also the slack matrix of the  {\em homogenization cone} of $P$:
\begin{equation}
	P^h=\RR_+^p\cdot [\onevec,V]=\setdef{(x_0,x)\in\RR\times\RR^n}{Wx\le x_0w}
\end{equation}
with respect to the representation 
$([\onevec,V],
{\tiny
	\left[
	\begin{array}{c}
		w^T \\ -W^T 
	\end{array}
	\right]
}
)$. 
Since $\dim(P)\ge 1$, there is some $c\in\RR^n$ with 
\begin{equation*}
	\max\setdef{c^Tx}{x\in P}-\min\setdef{c^Tx}{x\in P}=1\,,
\end{equation*}
and hence, due to LP-duality, we get
\begin{equation}\label{eq:zerovec1incone}
	(1, \zerovec^T)\in\RR^q\cdot(w,W)\, \textup{ and so also, } (1, \zerovec^T)\in\RR^q\cdot(w,-W).
\end{equation}
From~\eqref{eq:defslackpoly} and~\eqref{eq:zerovec1incone}
we get that $\onevec\in S\cdot\RR^q$, the column span of $S$. These properties characterize the slack matrices of polytopes of dimension at least one:

\begin{theorem}\label{thm:sm poly}
	A matrix $M\in\RR_+^{p\times q}$ with $\rank(M)\ge 2$ is a slack matrix of a polytope if and only if~$M$ is a slack matrix of a polyhedral cone and $\onevec\in M\cdot\RR^q$.
\end{theorem}

\begin{proof}
	It suffices to show that a matrix $M\in\RR_+^{p\times q}$  with $\onevec\in M\cdot\RR^q$ that is the slack matrix of some cone~$K\subseteq \RR^n$ with respect to a representation $(A,B)$ is also the slack matrix of some polytope. To construct such a polytope, choose any $\mu\in\RR^q$ such that $\onevec=M\mu$ and define $c=B\mu$. Then $Ac=\onevec$ since $M=AB$. Define $P=\convexhull(\rows{A})$.  Then we have:
	\begin{multline*}
		P= \setdef{y^TA}{y^T\onevec=1,y\in\RR_+^p}
			=\setdef{y^TA}{y^TAc=1,y\in\RR_+^p}  \\
			=\setdef{x\in K}{x^Tc=1}
			=\setdef{x\in\RR^n}{x^TB\ge\zerovec,x^Tc=1}\, .
	\end{multline*}
Mapping the hyperplane in $\RR^n$ defined by $x^Tc=1$ isometrically to $\RR^{n-1}$ (as in the proof of Theorem~\ref{thm:slackconesrows}), we find that~$M$ is a slack matrix of the resulting image of~$P$.
\end{proof}

\begin{corollary}\label{cor:sm poly}
	A matrix $M\in\RR_+^{p\times q}$ with $\rank(M)\ge 2$ is a slack matrix of some polytope if and only if it satisfies the RCGC (or, equivalently, the CCGC) and $\onevec\in M\cdot\RR^q$ holds.
\end{corollary}

Theorem~\ref{thm:slackconesrows} geometrically characterizes the slack matrices of cones as those matrices~$M \in \RR^{p \times q}_+$ that satisfy
\begin{equation}\label{eq:geometric RCGC}
	\conichull(\rows{M})=\lin(\rows{M})\cap\RR_+^q\, .
\end{equation}
There is an analogous geometric characterization of slack matrices of polytopes.

\begin{corollary}\label{cor:sm poly aff}
  A matrix $M\in\RR_+^{p\times q}$ with $\rank(M)\ge 2$ is a slack matrix of some polytope if and only if 
  \begin{equation}\label{eq:sm poly aff criterion}
    \convexhull(\rows{M})=\aff(\rows{M})\cap \RR^{q}_{+}\, .
  \end{equation}
\end{corollary}
\begin{proof}
First, suppose that $M$ is a slack matrix of some polytope.  Then by Corollary~\ref{cor:sm poly}, we have that $M$ satisfies~\eqref{eq:geometric RCGC} and $\onevec\in M\cdot\RR^q$.  Hence, there exists some $c \in \RR^q$ such that $Mc = \onevec$ and the affine hyperplane $L=\setdef{x\in \RR^q}{x^T c = 1}$ contains the rows of $M$.  Intersecting $L$ with both sides of~\eqref{eq:geometric RCGC}, we obtain~\eqref{eq:sm poly aff criterion}.
    
For the reverse implication,
 let $M\in\RR_+^{p\times q}$ be a nonnegative matrix satisfying~\eqref{eq:sm poly aff criterion}. 
	Using any isometry~$\varphi$ between the $d$-dimensional affine subspace $\aff(\rows{M})$ and~$\RR^d$, we find that~$M$ is a slack matrix of the $\varphi$-image of the polytope defined in~\eqref{eq:sm poly aff criterion}.
\end{proof}

We have seen above that every slack matrix of a polytope~$P$ has the all-ones vector in its column span and is also a slack matrix of the homogenization cone $P^h$ of~$P$. 
The next example shows that not all slack matrices of $P^h$ are slack matrices of~$P$, in fact, this does not even hold for the slack matrices of~$P^h$ that have the all-ones vector in their column span.

\begin{example}\label{ex:square slack}
Let $P$ be the square $[-1,1]^2$.  The matrix
\[ M = \left( \begin{array}{cccc}
\frac{4}{3} & 0 & \frac{4}{3} & 0 \\
2 & 0 & 0 & 2 \\
0 & 2 & 2 & 0 \\
0 & 4 & 0 & 4 \end{array} \right)
= \left( \begin{array}{ccc}
\frac{2}{3} & \frac{2}{3} & \frac{2}{3} \\
1 & 1 & -1 \\
1 & -1 & 1 \\
2 & -2 & -2 \end{array} \right)
\left( \begin{array}{cccc}
1 & 1 & 1 & 1 \\
1 & -1 & 0 & 0 \\
0 & 0 & 1 & -1 \end{array} \right) \]
is in $\mathcal{S}_{P^h}$ and $\onevec$ is in the column span of $M$.  It is clear, however, that $M$ is not in $\mathcal{S}_P$ since each facet of $[-1,1]^2$ is equidistant from the two vertices not on the facet.  On the other hand, since $M$ has the RCGC/CCGC and $\onevec$ is in its column span, it is the slack matrix of some other polytope $Q$. To obtain it, write a new rank factorization of $M$ (note that $\rank(M)=3$) so that the first factor contains the all ones vector as its first column as follows:
\[ M = \left( \begin{array}{ccc}
\frac{2}{3} & \frac{2}{3} & \frac{2}{3} \\
1 & 1 & -1 \\
1 & -1 & 1 \\
2 & -2 & -2 \end{array} \right)
UU^{-1}
\left( \begin{array}{cccc}
1 & 1 & 1 & 1 \\
1 & -1 & 0 & 0 \\
0 & 0 & 1 & -1 \end{array} \right), \,\,\,\,\,\,\,
U = \left( \begin{array}{ccc}
1 & 0 & 0 \\
1/4 & 1& 0\\
1/4 & 0 & 1 \end{array} \right)
 \]
 to get
 \[ M = \left( \begin{array}{cccc}
\frac{4}{3} & 0 & \frac{4}{3} & 0 \\
2 & 0 & 0 & 2 \\
0 & 2 & 2 & 0 \\
0 & 4 & 0 & 4 \end{array} \right)
= \left( \begin{array}{ccc}
1 & \frac{2}{3} & \frac{2}{3} \\
1 & 1 & -1 \\
1 & -1 & 1 \\
1 & -2 & -2 \end{array} \right)
\left( \begin{array}{cccc}
1 & 1 & 1 & 1 \\
3/4 & -5/4 & -1/4 & -1/4 \\
-1/4 & -1/4 & 3/4 & -5/4 \end{array} \right). \]
By Lemma~\ref{lem:rank factorization}, $M$ is the slack matrix of the cone with $\VV$-representation the rows of the first factor and $\HH$-representation the columns of the second factor. Assuming the coordinates of this three-dimensional cone are $x_0,x_1,x_2$, and slicing the cone with the hyperplane $\{ ((x_0,x_1,x_2) \,:\, x_0 = 1 \}$ gives a polytope 
$Q$ with vertices $(2/3,2/3), (1,-1), (-1,1),(-2,-2)$ and $\HH$-representation given by the columns of the second factor. 
Then $M \in \mathcal{S}_Q$.
\end{example}

\subsection{Further Results on Slack Matrices of Cones and Polytopes}\label{subsec:further results}

In this section we derive some more insight into the geometric relations between cones, polytopes, and their slack matrices that will be useful in later parts of the paper. We return to the setup used earlier: $K$ is assumed to be a cone and $S$ the slack matrix of~$K$ with respect to its representation $(A,B)$ where $A \in \RR^{p \times n}$ and $B \in \RR^{n \times q}$. 

First, we will show that every slack matrix of a cone is the slack matrix of some pointed cone.  Recall that we use $\lin(K)$ to denote the linear hull of~$K$ and $\lineal(K)$ to denote the lineality space of~$K$.  
Then we have	$\lin(K)=\RR^p\cdot A$ and $\lineal(K)=\leftkernel(B)$. A cone $K$ is {\em pointed} if  $\lineal(K) = \{\zerovec\}$.
Define
\begin{equation*}
	L:=\lin(K)\cap\lineal(K)^{\perp}=(\RR^p\cdot A)\cap(B\cdot\RR^q)\,.
\end{equation*}
Then we have
\begin{equation*}
	\lin(K)=L+\lineal(K)
\end{equation*}
(where the summands are orthogonal to each other)
and
\begin{equation*}
	K=(K\cap L)+\lineal(K)\,,
\end{equation*}
where $K\cap L\subseteq L$ is a pointed (i.e., having trivial lineality space) cone with $\dim(K\cap L)=\dim(L)$.
Denoting by $A'\in\RR^{p\times n}$ the matrix obtained from~$A$ by orthogonal projections of all rows to~$L$, we have 
\begin{equation*}
	K\cap L=\RR_+^p\cdot A'
	\qquad\text{and}\qquad
	S=A'B\,.
\end{equation*}
By mapping~$L$ isometrically to $\RR^{\dim(L)}$, we thus find that~$S$ is a slack matrix of the pointed cone that is the image of~$K\cap L$ under that map and we get the following:

\begin{lemma}\label{lem:sm cone iff sm pointed cone}
	A matrix is a slack matrix of a polyhedral cone if and only if it is a slack matrix of some pointed polyhedral cone.
\end{lemma}

If the cone~$K$ is pointed, then for every zero-row of~$S=AB$ the corresponding row of~$A$ is a zero-row as well. Hence, removing any zero-row from~$S$ results in another slack matrix of~$K$. A similar statement clearly holds for adding zero-rows.

\begin{lemma}\label{lem:zero rows sm pointed cone}
	If a matrix~$S$ is a slack matrix of a pointed polyhedral cone~$K$ then every matrix obtained from~$S$  by adding or removing zero-rows is a slack matrix of~$K$ as well.
\end{lemma} 

Lemmas~\ref{lem:sm cone iff sm pointed cone} and~\ref{lem:zero rows sm pointed cone}  together also imply this statement:

\begin{lemma}\label{lem:sm add rem rows}
	If a matrix is a slack matrix of some polyhedral cone then every matrix obtained from it by adding or removing zero-rows is a slack matrix of some polyhedral cone as well.
\end{lemma}

Let us further investigate the linear map $x\mapsto x^TB$. It induces the isomorphism
\begin{equation}\label{eq:iso L}
	L\xrightarrow[{\textup{isomorphism}}]{\cdot B}\RR^p\cdot S
\end{equation}
between the linear space $L$ and the row span of~$S$ 
because of the relations:
\begin{equation*}
	L\subseteq\lineal(K)^{\perp}=\leftkernel(B)^{\perp}
\end{equation*}
and
\begin{equation*}
	L\cdot B=(L+\lineal(K))\cdot B=\lin(K)\cdot B=(\RR^p\cdot A)\cdot B=\RR^p\cdot S\,.
\end{equation*}
It also induces the isomorphism
\begin{equation*}
	K\cap L\xrightarrow[{\textup{isomorphism}}]{\cdot B}\RR_+^p\cdot S
\end{equation*}
between the cone $K\cap L$ and the cone spanned by the rows of~$S$ 
since $(K\cap L)\cdot B=((K\cap L)+\lineal(K))\cdot B=K\cdot B=(\RR_+^p\cdot A)\cdot B=\RR_+^p\cdot S$. In particular, we have shown the following result:

\begin{lemma}\label{lem:pointed rank}
A polyhedral cone $K$ is pointed if and only if $\dim(K)=\rank(S)$ for any slack matrix $S$ of $K$.
\end{lemma}

Recall that if $P$ is a polytope with representation $(V,W,w)$ and slack matrix $S=[\onevec,V]\cdot B$ where 
\begin{equation*}
	B=\left[
		\begin{array}{c}
			w^T \\ -W^T 
		\end{array}
	\right],\,
\end{equation*}
then the homogenization $P^h$ of~$P$ is a pointed cone that also has $S$ as a slack matrix. Since $P^h$ is pointed, $L$ contains the entire cone and we can restrict the isomorphism in~\eqref{eq:iso L} to the set $\{1\}\times P=\convexhull(\rows{[\onevec,V]})$.  Thus we have that $\{1\}\times P$ is isomorphic to $\convexhull(\rows{[\onevec,V]})\cdot B=\convexhull(\rows{S})$. This establishes the first part of the following:

\begin{theorem}\label{thm:P iso conv rows S}
	If~$S$ is a slack matrix of the polytope~$P$, then $P$ is isomorphic to $\convexhull(\rows{S})$. In addition, we have $\dim(P)=\rank(S)-1$.
\end{theorem} 

\begin{proof}
To prove the second statement, note that $\dim(P^h) = \dim(P)+1$.  By Lemma~\ref{lem:pointed rank}, we have that $\dim(P^h)=\rank(S)$.
\end{proof}

In the conic case, we had that $M \in \mathcal{S}_K$ if and only if $M^T \in \mathcal{S}_{K^*}$.  This correspondence breaks down for polytopes as we see in the example below.  The reason behind this is that we cannot scale $\VV$-representations of polytopes by positive scalars.

\begin{example}\label{ex:triangular prism}
The matrix
\[ M = \left( \begin{array}{ccccc}
1 & 1 & 0 & 0 & 0 \\
1 & 0 & 1 & 0 & 0 \\
1 & 0 & 0 & 1 & 0 \\
0 & 1 & 0 & 0 & 1 \\
0 & 0 & 1 & 0 & 1 \\
0 & 0 & 0 & 1 & 1 \end{array} \right) \]
is a slack matrix for the triangular prism in $\RR^3$. Thus, by Corollary~\ref{cor:sm poly}, $M$ satisfies both the RCGC and the CCGC, and the all ones-vector is in the column span of~$M$. However, 
the all-ones vector is not in the row span of $M$, so $M^T$ is not the slack matrix of any polytope.
\end{example}

Despite this complication, we can still derive some results for transposes of slack matrices of polytopes.  Recall that the {\em polar} of a polytope $P\subset \RR^n$ is 
\[ P^\circ = \setdef{y\in \RR^n}{x^T y \leq 1 \textup{ for all } x \in P}\, . \]
Then $P^\circ$ is a polytope whenever $0\in \textup{int}(P)$, the interior of $P$. Since translating $P$ does not change its slack matrices, we may assume that $0\in \textup{int}(P)$. Therefore, $P$ has an $\HH$-representation of the form 
$P = \{ x \in \RR^n \,:\, Wx \leq \onevec \}$ and $P^\circ = \convexhull(\textup{rows}(W))$. Similarly,  if $P=\convexhull(\textup{rows}(V))$, then $P^\circ = \setdef{x \in \RR^n}{Vx \leq \onevec}$.  This implies that the slack matrix of $P$ with respect to the representation $(V,W,\onevec)$ is the transpose of the slack matrix of $P^\circ$ with respect to the representation $(W,V,\onevec)$ and we get the following result that is analogous to Proposition~\ref{prop:transpose for cones} for cones.

\begin{proposition} \label{prop:polar1}
For any polytope $P$, there exists a slack matrix $M \in \mathcal{S}_P$ such that $M^T$ is also a slack matrix of a polytope.
\end{proposition}

In the light of Theorem~\ref{thm:sm poly}, this says that slack matrices of polytopes (which already have $\onevec$ in their column span) allow positive scalings of their columns that puts $\onevec$ into their row span as well. This is false for 
general nonnegative matrices.

\begin{example} Continuing Example~\ref{ex:triangular prism}, we see that the following matrix $M'$ obtained by 
scaling the columns of $M$ is also a slack matrix of the same prism and does have $\onevec$ in its row span:
\[ M' = \left( \begin{array}{ccccc}
2 & 2 & 0 & 0 & 0 \\
2 & 0 & 4 & 0 & 0 \\
2 & 0 & 0 & 4 & 0 \\
0 & 2 & 0 & 0 & 2 \\
0 & 0 & 4 & 0 & 2 \\
0 & 0 & 0 & 4 & 2 \end{array} \right). \]
The prism has vertices: $$(0,1,-1),(2,-1,-1),(-2,-1,-1),(0,1,1),(2,-1,1),(-2,-1,1)$$ and $M'$ comes from the facet 
description: 
$$z \leq 1, -y \leq 1, -x+y \leq 1, x+y \leq 1,  -z \leq 1.$$
Therefore, $P^\circ$ has vertices $(0,0,1), (0,-1,0),(-1,1,0),(1,1,0),(0,0,-1)$ and is a bisimplex with slack matrix 
$M'^T$.
\end{example}

We can also show a converse to Proposition~\ref{prop:polar1}. 

\begin{proposition} \label{prop:polar2}
Suppose $M\in \RR_+^{p\times q}$ such that $M$ and $M^T$ are both slack matrices of polytopes.  Then there exists a polytope $P$, with $0 \in \textup{int}(P)$, such that $M \in \mathcal{S}_P$ and $M^T \in \mathcal{S}_{P^\circ}$.
\end{proposition}

\begin{proof}
Since $M^T$ is a slack matrix of a polytope, we have that $\onevec \in \RR_+^p \cdot M$.  Without loss of generality, we can scale $M$ by a positive scalar so that $\onevec \in \convexhull(\rows{M})$.  

Let $M$ be a slack matrix of a polytope $R$ with $\textup{dim}(R)=d$. By Theorem~\ref{thm:P iso conv rows S}, 
$\rank(M)=d+1$.  Since the convex hull of the rows of $M$ is isomorphic to $R$, we have that the affine hull of the rows of $M$ has dimension $d$.  Let $J$ denote the all-ones matrix of dimension $p \times q$.  Since $\onevec$ is contained in the affine hull of the rows of $M$, we have that the affine hull of the rows of $M-J$ passes through the origin and has dimension $d$.  Hence, $\rank(M-J)=d$. This implies that we can write $M-J=AB$ with $A \in \RR^{p \times d}$ and $B \in \RR^{d \times q}$.  

Let $A'= \left( \onevec , A  \right)$ and let $B'= \left(\onevec, B^T \right)^T$.  Then $M=A'B'$ is a rank factorization of $M$. Let $P := \convexhull(\textup{rows}(A))$ and $Q := \{ x \in \RR^d \,:\, \onevec + x^TB \geq \zerovec\}$. 
Then the rows of $A'$ form a $\VV$-representation of $P^h$ and the columns of $B'$ form a $\HH$-representation for $Q^h = \{ (x_0,x) \in \RR^{d+1} \,:\, \onevec x_0 + x^TB \geq \zerovec\}$. By Lemma~\ref{lem:rank factorization}, $P^h=Q^h$ which implies that 
$P=Q$. Therefore, $M$ is a slack matrix of $P$ and $M^T$ is a slack matrix of $P^\circ$.
\end{proof}

\section{An Algorithm to Recognize Slack Matrices} \label{sec:algorithm}
In this section, we discuss the algorithmic problem of deciding whether a given nonnegative matrix has the RCGC (or, equivalently, the CCGC). According to Corollaries~\ref{cor:sm ccgc rcgc} and~\ref{cor:sm poly} this is the crucial step to be performed in order to decide whether a given matrix is a slack matrix of a cone or a polytope. 

We start with a promising result:
\begin{theorem}
	The problem to decide whether a nonnegative matrix satisfies the RCGC (or the CCGC) is in coNP. In particular, the same holds for checking the property of being a slack matrix (of a cone or of a polytope).
\end{theorem}

\begin{proof}
	If the given matrix $M\in\RR_+^{p\times q}$ does not satisfy the RCGC, then there is some point $x\in\RR^p\cdot M\cap\RR_+^q\setminus\RR_+^p\cdot M$ (which can be chosen to have coordinates whose encoding lengths are bounded polynomially in the encoding length of~$M$). The fact that $x\not\in\RR_+^p\cdot M$ can be certified by the help of some separating hyperplane whose normal vector can be chosen to have coordinates with encoding length bounded polynomially in the encoding length of~$M$ as well.
\end{proof}

Next, we are going to describe an algorithm to check the CCGC (equivalently, the RCGC) for a nonnegative matrix.  By Corollary~\ref{cor:sm ccgc rcgc}, this algorithm will then provide a method to check if a given nonnegative matrix is a slack matrix of a cone.  To check if the matrix is the slack matrix of a polytope (see Corollary~\ref{cor:sm poly}), we can add the additional step of checking if the all-ones vector is in the column span of the matrix which is doable in polynomial time.  A SAGE worksheet implementing this code can be found at {\tt http://www.math.washington.edu/{\raise.17ex\hbox{$\scriptstyle\sim$}}rzr}.

\vspace{.2cm}

\centerline{\bf Algorithm to check if a nonnegative matrix has the CCGC}
\vspace{.1cm}
\noindent{\bf Input}:  A matrix $M  \in \RR_+^{p\times q}$.\\
\noindent{\bf Output}: {\tt True} if $M$ has the CCGC and {\tt False} otherwise.
\begin{enumerate}
\item Compute a basis $L$ for the left kernel of $M$.  For each vector $\ell$ in $L$, generate the equation $\ell^T x =0$.
\item Generate an $\HH$-representation of the cone~$K$ with the equations from the previous step and the inequalities $x_1\geq 0,\ldots,x_p \geq 0$.
\item Compute a minimal $\VV$-representation of $K$.
\item Normalize the vectors in the $\VV$-representation and the columns of $M$.
\item Check that each normalized vector in the $\VV$-representation is a normalized column of $M$. If so, return {\tt True}.  If not, return {\tt False}.
\end{enumerate}

\begin{proof}
	We have $K=M\cdot\RR^q\cap\RR_+^p$ and $M\cdot\RR_+^q\subseteq K$ due to the nonnegativity of~$M$. Thus,  $M$ satisfies
the CCGC if and only if $K\subseteq M\cdot\RR_+^q$ holds, which is what the algorithm checks in the last three steps (note that all cones involved are pointed because they are contained in $\RR_+^p$).
\end{proof}

The only computationally challenging part of the algorithm is converting from the $\HH$-representation of $K$ to a $\VV$-representation.  There are several algorithms to do this, and we refer to \cite{Joswig}, \cite{Matheiss}, and \cite{Seidel} for information on the different techniques.  No polynomial time algorithm for this conversion exists, since the $\VV$-representation may have size exponential in that of the $\HH$-representation.  If the  dimension of the cone
is fixed, however, then there do exist polynomial time algorithms for the conversion \cite{Dyer}. Thus, we obtain the following complexity results.

\begin{theorem}
	For fixed~$r$, checking whether a  rank $r$ matrix satisfies the RCGC (CCGC) can be done in polynomial time. In particular, checking whether matrices of fixed rank are slack matrices of cones or polytopes can be done in polynomial time.  
\end{theorem}

Given an $\HH$-polyhedron  $P$ and a $\VV$-polytope $Q$ contained in $P$, the problem of deciding whether $P=Q$ is known as the \emph{polyhedral verification problem}. The complexity of this problem is unknown \cite{Seymour}. However,  a polynomial time algorithm for the polyhedral verification problem would yield an \emph{output sensitive} algorithm for the problem of computing the facets of a polytope given in $\VV$-representation, and thus solve a decades old open problem in computational geometry (see~\cite{JZ04}). 

Clearly,  given a $\VV$-polytope it is easy to check whether it is contained in an $\HH$-polyhedron. The reverse problem of checking whether an $\HH$-polyhedron is contained in a $\VV$-polytope is known to be coNP-complete \cite{FreundOrlin}. Note that the polyhedral verification problem is the restriction of the latter problem to those instances in which the $\VV$-polytope is contained in the $\HH$-polyhedron (see also \\ \url{http://www.inf.ethz.ch/personal/fukudak/polyfaq/node21.html}, \cite{KaibelPfetsch} and~\cite{KP03}).

\begin{theorem}\label{thm:polytop verification}
	The following problems can be reduced in polynomial time to each other:
	\begin{enumerate}
		\item The polyhedral verification problem
		\item Is a given matrix a slack matrix of a polytope?
		\item Is a given matrix a slack matrix of a cone?
		\item Does a given matrix satisfy the RCGC/CCGC?
	\end{enumerate}
\end{theorem}

\begin{proof}
	Corollary~\ref{cor:sm poly} shows that~(2) can be reduced (in polynomial time) to (4) (since checking whether $\onevec$ is contained in the column space can be done in polynomial time) and Corollary~\ref{cor:sm ccgc rcgc} shows that (4) can be reduced to (3).  
	
	We can also reduce (3) to (2): Suppose we need to check whether a given matrix~$M$ is a slack matrix of a cone.  By Lemma~\ref{lem:zero rows sm pointed cone}, we can assume that $M$ has no zero rows.  We can also scale the rows of $M$ by positive scalars without effect on $M$ being a slack matrix of a cone.  Using these two facts, we can assume that $\onevec$ is in the column span of $M$.  Then, being a slack matrix of a cone is equivalent to being a slack matrix of a polytope due to Theorem~\ref{thm:sm poly}.
	
	Since Corollary~\ref{cor:sm poly aff} shows how to reduce (2) to (1), it thus remains to establish a reduction of (1) to~(2).  Let $Q=\convexhull(\rows{V})$ with $V\in\RR^{p\times n}$ and $P=\setdef{x\in\RR^n}{Wx\le w}$ with $W\in\RR^{q\times n}$ and $w\in\RR^q$ with $Q\subseteq P$. Suppose we need to decide whether $P= Q$. First, we check whether~$P$ is pointed (i.e., $W$ has a trivial right kernel) and $\dim(P)=\dim(Q)$ (both checks can be done in polynomial time, the second one using linear programming). If either check fails, then $P\ne Q$. 
	
	So let us assume $\dim(P)=\dim(Q)$  and that~$P$ is pointed. The latter fact implies that the affine map $\varphi:\RR^n\rightarrow\RR^q$ defined via $\varphi(x)=w-Wx$ is injective. Let $M$ be the matrix arising from~$V$ by applying~$\varphi$ to each row. Then, due to $Q\subseteq P$, we have that~$M$ is nonnegative. According to Corollary~\ref{cor:sm poly aff}, the matrix $M$ is a slack matrix of a polytope if and only if 
	\begin{equation}\label{eq:polytop verification}
		\convexhull(\rows{M})=\aff(\rows{M})\cap \RR_+^q.
	\end{equation}
	Since we have 
	\begin{equation*}
		\convexhull(\rows{M})=\varphi(\convexhull(\rows{V}))=\varphi(Q)
	\end{equation*}
	and  
	\begin{multline*}
		\aff(\rows{M})\cap \RR^{q}_{+}
		=\varphi(\aff(\rows{V}))\cap\RR_+^q
		=\varphi(\aff(Q))\cap\RR_+^q\\
		=\varphi (\setdef{x\in\aff(Q)}{\varphi(x)\ge\zerovec})
		=\varphi(P\cap\aff(Q))=\varphi(P)
	\end{multline*}
	(here we used that $\dim(P)=\dim(Q)$), condition~\eqref{eq:polytop verification} is equivalent to $\varphi(P)=\varphi(Q)$.
	In turn, this is equivalent to $P=Q$ since $\varphi$ is injective. Thus, $P=Q$ is equivalent to $M$ being the slack matrix of a polytope.
\end{proof}

\section{A Combinatorial Characterization of Slack Matrices}\label{sec:comb}

Our second characterization of slack matrices of  cones and polytopes relies on incidence structures. For a (nonnegative) matrix $M$, we denote by $\inc{M}$ the 0/1-matrix with $(\inc{M})_{ij}=1$ if and only if $M_{ij}=0$. The matrices $\inc{M}$ arising from slack matrices~$M$ of a polyhedral cone~$K$ or of a polytope~$P$ are called the \emph{incidence matrices} of~$K$ or $P$, respectively.

We start by characterizing the slack matrices of polytopes, since  the corresponding statement for cones can easily be deduced from the one for polytopes. The characterization is restricted to nonnegative matrices of rank at least two.  It is easy to see that no matrix of rank one is a slack matrix of a nontrivial polytope. One may (or may not) want to consider a rank-zero matrix as a slack matrix of the polytope consisting of the zero-vector in $\RR^0$.

\begin{theorem}\label{thm:incidences_polytopes}
	A nonnegative matrix $M$ with $\rank(M)\ge 2$ is a slack matrix of some  polytope if and only if $\inc{M}$ is an incidence matrix of some $(\rank(M)-1)$-dimensional polytope and $\onevec$ is contained in the column span of~$M$.
\end{theorem}

\begin{proof} If~$M$ is a slack matrix of a polytope~$P$, then $\onevec$ is contained in the column span of~$M$ (Theorem~\ref{thm:sm poly}), and by Theorem~\ref{thm:P iso conv rows S}, $\dim(P)=\rank(M)-1$.
	
	In order to establish the non-trivial implication of the claim, let $M\in\RR_+^{p\times q}$ be a nonnegative matrix with $\rank(M) = d+1 \ge 2$, $\onevec \in M \cdot \RR^q$ and $\inc{M}$ an incidence matrix of some $d$-dimensional polytope~$R$. Denote by $V\subseteq\RR_+^q$ the set of rows of~$M$ and define the polytope $P :=\convexhull(V)$ and the polyhedron 
	$Q:=\aff(V)\cap\RR_+^q$.
	Clearly, $P\subseteq Q$, and since $\onevec \in M \cdot \RR^q$, $\dim(Q)=\dim(P)=d$.
	By Corollary~\ref{cor:sm poly aff}, in order to show that $M$ is a slack matrix of a polytope, it suffices to prove $P=Q$.
	
	In order to establish $Q\subseteq P$, let us define
	\begin{equation*}
		V_i=\setDef{v\in V }{v_i=0}
		\quad\text{and}\quad
		F_i=\convexhull(V_i)
		\quad\text{for }1\le i\le q\,.
	\end{equation*}
	The set 
	\begin{equation*}
		F=\bigcup_{i=1}^q F_i
	\end{equation*}
	 is contained in the relative boundary $\partial{Q}$ of~$Q$. Note that as an incidence matrix of some polytope of dimension at least one, $M_{\text{inc}}$ does not have an all-ones column.  Since $Q=\convexhull(\partial Q)$ (note that~$Q$ is a \emph{pointed} polyhedron of dimension  $d\ge 2$, which is important here in case of~$Q$ being unbounded), if we show that $F= \partial Q$, then we will have that $Q = \convexhull(F)\subseteq P$. 

	Thus, our goal is to establish $F=\partial Q$. As mentioned above, we have $F\subseteq\partial Q$. It suffices to show that $F$ is homotopy-equivalent to a $(d-1)$-dimensional sphere\footnote{Our proof of this is inspired by~\cite{JZ04}.}, because then~$F$ cannot be \emph{properly} contained in the $(d-1)$-dimensional connected (recall $\dim(Q)\ge 2$) manifold $\partial Q$. This follows, e.g., from \cite[Cor.~8.5]{Bre93} together with the fact that the $(d-1)$-st cohomology group of a $(d-1)$-dimensional sphere is non-trivial.
	
	To show that $F$ is homotopy-equivalent to a $(d-1)$-dimensional sphere, observe that for every subset $I\subseteq\{1,\dots,q\}$, we have $\cap_{i\in I}F_i\ne\varnothing$ if and only if the submatrix of $\inc{M}$ formed by the columns indexed by~$I$ has an all-ones row.
Now let $R$ be a polytope of which $\inc{M}$ is an incidence matrix.  Let $G_1,\dots,G_q$ be the faces of $R$ that correspond to the columns of $\inc{M}$.  Then $\cap_{i\in I}G_i\ne\varnothing$ holds if and only if the submatrix of $\inc{M}$ formed by the columns indexed by~$I$ has an all-ones row. 

 Therefore, the abstract simplicial complexes
		\begin{equation*}
			\setDef{I\subseteq\{1,\dots,q\}}{\bigcap_{i\in I}F_i\ne\varnothing}, \,\,\textup{ and } \,\,
			\setDef{I\subseteq\{1,\dots,q\}}{\bigcap_{i\in I}G_i\ne\varnothing}
		\end{equation*}
		(known as the \emph{nerves} of the polyhedral complexes induced by $F_1,\dots,F_q$  and by $G_1,\dots,G_q$, respectively) are identical. Since all intersections $\bigcap_{i\in I}F_i$ and $\bigcap_{i\in I}G_i$ are contractible (in fact, they are even convex), this simplicial complex is homotopy equivalent to both~$F$ and to the $(d-1)$-dimensional (polyhedral) sphere~$\partial{R}$ (see, e.g., \cite[Thm.~10.6]{Bjo95}).
	\end{proof}

Since polygons have a very simple combinatorial structure, Theorem~\ref{thm:incidences_polytopes} readily yields a simple characterization of their slack-matrices. Here, a \emph{vertex-facet slack matrix} of a polytope~$P$ is a slack matrix of~$P$ whose rows and  columns are in one-to-one correspondence with the vertices and facets of~$P$, respectively.

\begin{corollary}
	A matrix~$M\in\RR_+^{n\times n}$  ($n\ge 3$) is a vertex-facet slack matrix of an $n$-gon if and only if its rows span an affine space of dimension exactly two and its rows and columns can be permuted such that the non-zero entries appear exactly at the positions $(i,i)$ (for $1\le i\le n$), and $(i,i-1)$ (for $2\le i\le n$), and $(1,n)$.
\end{corollary}

Steinitz' theorem~\cite{SteinitzRademacher} says that a graph~$G$ is the $1$-skeleton of a three-dimensional polytope if and only if~$G$ is planar and three-connected. Using this, one can check in polynomial time whether a given 0/1-matrix is an incidence matrix of a three-dimensional polytope. For every fixed $d\ge 4$, however, it is NP-hard to decide whether a given 0/1-matrix is an incidence matrix of a $d$-dimensional polytope~\cite{RichterGebert}.

In the following combinatorial characterization of slack matrices of cones we restrict our attention to matrices of rank at least two as for polytopes. Clearly, every nonnegative matrix of rank one is a slack matrix of the  ray $\RR_+^1$, and, we may consider a matrix of rank zero as a slack matrix of the trivial cone $\{0\}$ in $\RR^0$.

\begin{theorem}\label{thm:incidences_cones}
	A nonnegative matrix $M$ with $\rank(M)\ge 2$ is a slack matrix of a polyhedral cone if and only if $\inc{M}$ is an incidence matrix of some $\rank(M)$-dimensional pointed polyhedral cone.
\end{theorem}

\begin{proof}
	If~$M$  is a slack matrix of some polyhedral cone then,  by Lemma~\ref{lem:sm cone iff sm pointed cone}, $M$ is a slack matrix (and hence $\inc{M}$ is an incidence matrix) of a pointed  polyhedral cone~$K$. By Lemma~\ref{lem:pointed rank} this cone has dimension $\rank(M)$.
	
	In order to prove the reverse implication, we can assume by the results in Section 2.3 that~$M$ does not have any zero-row. Since $M$ is also nonnegative, there exists a positive diagonal matrix $D$ such that $DM$ contains $\onevec$ in its column span.

Given a pointed cone $K$, we can slice $K$ by an affine hyperplane $L$ such that the slice is a polytope of dimension $\dim(K)-1$ and the incidence structures of $K$ and $K\cap L$ are identical.  Thus, $\inc{(DM)}$ is an incidence matrix of some $(\rank(M)-1)$-dimensional polytope.  By Theorem~\ref{thm:incidences_polytopes}, we have that $DM$ is a slack matrix of a polytope.  Hence, $M$ is a slack matrix of the homogenization cone of this polytope.
\end{proof}

Note that dropping \emph{pointed} from the formulation of Theorem~\ref{thm:incidences_cones} makes the statement false. Indeed, 
\begin{equation*}
	M=
	\left[
	\begin{array}{cc}
		1 & 2 \\
		2 & 1 \\
		0 & 0 \\
		0 & 0
	\end{array}
	\right]
	\quad\text{with}\quad
	\inc{M}=
	\left[
	\begin{array}{cc}
		0 & 0 \\
		0 & 0 \\
		1 & 1 \\
		1 & 1
	\end{array}
	\right]
\end{equation*} 
and $\rank(M)=2$ is not a slack matrix (since~$M$ does not satisfy the RCGC), but $\inc{M}$ is the incidence matrix of the non-pointed cone $\{(x_1,x_2):x_2\ge 0\}$ with $\mathcal{V}$-representation $(0,1)$, $(0,1)$, $(1,0)$, $(-1,0)$ and $\mathcal{H}$-representation $(0,1)$, $(0,1)$.

\bibliographystyle{plain}

\end{document}